\documentclass[11pt,reqno]{amsart}
\usepackage{amsaddr}
\usepackage{graphicx}
\usepackage{amsmath,amsthm,amssymb,amsfonts}
\usepackage{hyperref}
\usepackage[noadjust]{cite}
\usepackage[a4paper, left=2.5cm, right=2.5cm, top=2cm, bottom=2cm]{geometry}
\newtheorem*{example}{Examples}

\begin{document}

\title[\hfilneg 2025/10\hfil Fourth-order MEMS/NEMS models II]
{Global bifurcation curves for fourth-order MEMS/NEMS models II}
\author[M. Lin \& H. Pan\hfil 2025/10\hfilneg]{Manting Lin and Hongjing Pan$^*$} 
         \address{School of Mathematical Sciences,
	South China Normal University,\\ Guangzhou 510631, P. R. China}
  \email{linmt@m.scnu.edu.cn \& panhj@m.scnu.edu.cn}
     \thanks{${}^*$Corresponding author.}


\begin{abstract}
Global solution curve and exact multiplicity of positive solutions for a class of fourth-order
semilinear equations with clamped boundary conditions are derived.
The results extend a theorem of P. Korman (2004) by allowing the presence of a singularity in the nonlinearity.
The paper also establishes an a priori estimate for $C^{3}$-norm of positive solutions, which is optimal in H\"{o}lder regularity.
Applications to MEMS/NEMS models are presented.
\end{abstract}

\maketitle
\numberwithin{equation}{section}
\newtheorem{theorem}{Theorem}[section]
\newtheorem{remark}[theorem]{Remark}
\newtheorem{lemma}[theorem]{Lemma}
\newtheorem{definition}[theorem]{Definition}

\section{Introduction}\label{sec:1}
Consider the following fourth-order equation with doubly-clamped boundary conditions
\begin{equation}\label{eq:4order}
\left\{\begin{array}{l}
u^{\prime \prime \prime \prime}(x)=\lambda f(u(x)), \quad x \in(0,1), \\
u(0)=u(1)=u^{\prime}(0)=u^{\prime}(1)=0,
\end{array}\right.
\end{equation}
where $\lambda>0$ is a parameter and $f$ is a non-negative continuous function on a interval $[0,r)$ with some $r\leq \infty$.  Problem \eqref{eq:4order} arises in many physical models
describing the deformation of elastic objects clamped at the endpoints. The nonlinearity $f$ represents a nonlinear external force.

In this paper, we are concerned with the structure of the solution set of positive solutions of \eqref{eq:4order}.
By a \emph{positive solution} we mean that a function $u\in C^4[0,1]$ satisfies \eqref{eq:4order} and $0\not\equiv u\geq 0$ on $(0,1)$.
Since we are only interested in positive solutions of \eqref{eq:4order}, we may assume, if necessary, that $f$ has been extended to a non-negative continuous function on $(-\infty,r)$.
Given $f$ is non-negative, by \cite[Theorem 1]{chow1973maximum}, any non-trivial solution of \eqref{eq:4order} is positive for $\lambda>0$ and no positive solutions occur for $\lambda<0$.
For each $\lambda>0$, we denote the corresponding solutions of \eqref{eq:4order} of \eqref{eq:4order} by $u_{\lambda}$ or $u$ in short.
Denote the solution set or the solution curve by
\begin{equation*}
\mathcal{S}=\left\{(\lambda,u)\mid\lambda>0 \text{ and } u \text{ is a positive solution of } \eqref{eq:4order}_\lambda\right\}.
\end{equation*}
We focus on the case where the nonlinearity $f$ has a singularity. Such problems are practically significant, especially in models of  Micro/Nano-Electro-Mechanical Systems (MEMS/NEMS). For instance, the typical example $f(u)=\frac{1}{(1-u)^2}$, which follows the inverse square law,  models the Coulomb force in the parallel plate capacitors. The recent monograph \cite{Koochi2020} presents various 1-D beam-type MEMS/NEMS models that conform to the equation in \eqref{eq:4order}. Notably, the singularity of $f$ among them is a critical characteristic, as demonstrated in examples following Theorem \ref{th2:u bounded} below. Those models have stimulated our research interest.

In the past two decades, fourth-order MEMS/NEMS models have attracted significant attention, and various numerical and theoretical results have been established (\cite{Esposito2010,Koochi2020,Laurencot2017,Lin2007,pelesko2002modeling,Cassani2009,Laurencot2014,Liu2022,Lindsay2014,Cowan2010,Guo2014,Guo2008,Guo2009}).
However, to the best of the authors' knowledge,  the existing results are still some distance away from characterizing the complete structure of the solution set
of problem \eqref{eq:4order}  when  $f$ exhibits a singularity at a certain point $0<r<\infty$.
Significant progress in this direction was made in [12] regarding the following problem
\begin{equation}\label{eq:biharmonic}
\Delta^{2} u-T \Delta u=\frac{\lambda}{(1-u)^{2}} \quad \text { in }  B_1, \quad u=\partial_n u=0 \quad \text {on }  \partial B_1,
\end{equation}
where  $T \geq 0$ and $B_1 \subset {\mathbb R}^{d}\ (d=1,2)$ is the unit ball.
A continuous global solution curve was established in \cite[Theorem 1.1]{Laurencot2014} by the real analytic bifurcation theory (\cite{Buffoni2000}). Furthermore, the endpoint of the curve was derived. However, the complete structure of the solution set as well as the exact multiplicity of solutions remains to be explored. In the present paper, we will focus on the 1-D problem \eqref{eq:4order} but with more general $f(u)$, which includes a variety of examples coming from MEMS/NEMS models in \cite{Koochi2020} (see Examples after Theorem \ref{th2:u bounded} below).

This paper can be considered a sequel of M.~Liu and the second author \cite{Liu2022}, where they obtained the global solution curve for problem \eqref{eq:4order} but with doubly-pinned boundary conditions
$$u(0)= u(1) = u^{\prime\prime}(0)  = u^{\prime\prime}(1)=0. $$
 In contrast, the present problem is more challenging for two main reasons: the maximum principle is not directly applicable and the concavity of positive solutions changes over the interval $(0,1)$.

B.~P. Rynne \cite{Rynne04} studied a $2m$-th Dirichlet boundary value problem, which includes \eqref{eq:4order} with $r=\infty$,
and under monotonicity and convexity assumptions, showed that the problem has a smooth solution curve $\mathcal{S}_0$ emanating from $(\lambda, u)=(0, 0)$, and described the possible shapes and asymptotics of the curve. For problem \eqref{eq:4order} with convex increasing nonlinearity $f$ defined on $[0,\infty)$,
P.~Korman \cite[Theorem 1.1]{Korman2004} first
proved the complete structure of the solution set,
using a bifurcation approach. The result is as follows.

\begin{theorem}[\cite{Korman2004}]\label{thm:korman}
	Assume that $ f(u) \in C^{2}(0,\infty) \cap C^1[0,\infty)  $ satisfies $f(u)>0$ for $u \geq 0$, $f^{\prime}(0) \geqslant  0$, $f^{\prime \prime}(u) >0$ for $u>0$, and
	\begin{align}
	\label{eq:supper}
	& \lim _{u \rightarrow \infty} \frac{f(u)}{u}=\infty,
	\end{align}
	Then all positive solutions of \eqref{eq:4order} lie on a unique smooth curve of solutions. This curve starts at $(\lambda, u)=(0,0) $,  it continues for $\lambda> 0$ until a critical $\lambda_{0}$, where it bends back, and continues for decreasing $\lambda$ without any more turns, tending to infinity when $\lambda\downarrow 0$.
	In other words, we have exactly two, one or no solutions, depending on whether $0<\lambda<\lambda_{0}, \lambda=\lambda_{0}$, or $\lambda>\lambda_{0}$.
	Moreover, all solutions are symmetric with respect to the midpoint $x=\frac{1}{2}$, and the maximum value of the solution, $ u(\frac{1}{2}) $, is strictly monotone on the curve.
\end{theorem}

The bifurcation diagram is depicted in Figure \ref{fig:fig1}(i).
The typical examples of $f$ satisfying the assumptions of Theorem \ref{thm:korman} are the exponential nonlinearity $f(u)=e^u$ and the power nonlinearity $f(u)=(1+u)^p\ (p>1)$. However, if $f$ is singular at some $r\in(0,\infty)$, for example, $f(u)=\frac{1}{(r-u)^{p}} \ (p>0)$,  then Theorem \ref{thm:korman} is no longer applicable.

Our goal in this paper is to establish a global bifurcation result, analogous to Theorem \ref{thm:korman}, for the case where $f$ has a singularity. To this end, we need to replace the superlinear condition \eqref{eq:supper} at infinity with a growth condition near singularity; see \eqref{eq:supper2} below.
Our main result is as follows.

	\begin{theorem}\label{th2:u bounded}
	Assume that $ f \in C^{2}[0, r)$ with some $r\in (0,\infty)$ satisfies
	\begin{align}\label{eq:f positive}
		& f(u)>0 \quad \text{for  } 0 \leq u<r,
		\\\label{eq:f increasing}
		& f^{\prime}(u) >  0 \quad \text{for } 0<u<r,
		\\\label{eq:convex}
		& f^{\prime \prime}(u) >0  \quad \text{for } 0<u<r,
		\\\label{eq:supper2}
		& 0<\liminf_{u \rightarrow r^{-}} \,(r-u)f(u) \leq \infty.
	\end{align}
	Then all conclusions of Theorem \ref{thm:korman} still hold, except that the solution curve finally tends to
a bounded function $w$ instead of infinity when $\lambda  \downarrow  0$.
Furthermore, the function $w$ belongs to $( {C}^{2+\alpha}[0,1] \setminus {C}^{3}[0,1])\cap C^{4}( [0,1]\setminus  \{  \frac{1}{2}\}  )$ for any $\alpha\in(0,1)$ and is an explicitly given axisymmetric function with respect to $x=\frac{1}{2}$, with the maximum value $w( \frac{1}{2})  = r$.
\end{theorem}

    The bifurcation diagram is depicted in Figure \ref{fig:fig1}(ii). The explicit expression of the singular solution $w$ is present in Lemma 2.3 below.
	\begin{figure}
	\centering
	\includegraphics[width=0.80\linewidth]{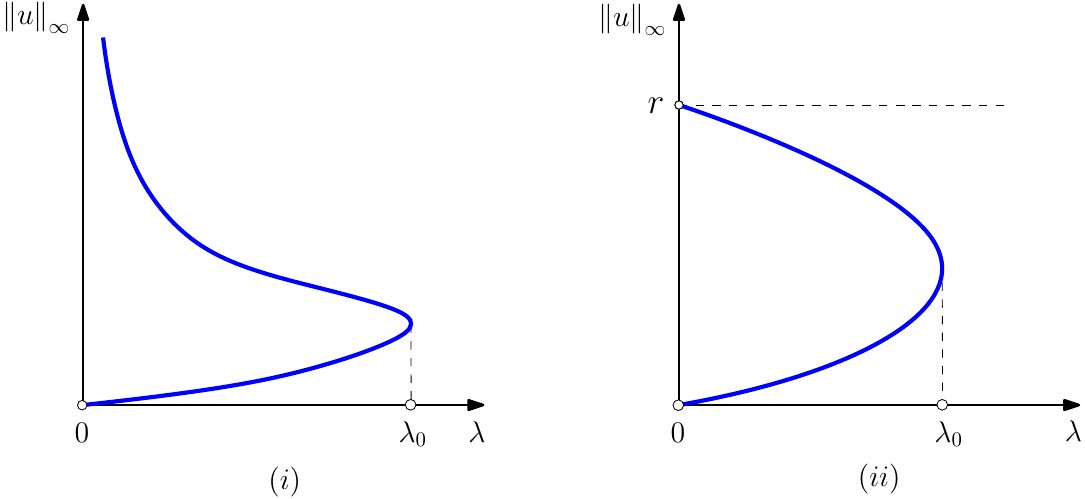}
	\caption{Global bifurcation diagrams provided by Theorems \ref{thm:korman} and \ref{th2:u bounded}. (i) $r=+\infty$ and $ \lim _{u \rightarrow +\infty} \frac{f(u)}{u}=+\infty$. (ii) $r<+\infty$ and $ \liminf_{u \rightarrow r^{-}} \,(r-u)f(u) >0 $.}
	\label{fig:fig1}
\end{figure}	

	The method adopted in our study is the bifurcation approach to fourth-order Dirichlet problems, originally formulated by P.~Korman \cite{Korman2004}.
	However, the singularity of $f$ poses new challenges.
  To overcome difficulties in applying the method, we establish the crucial a priori bound $\|u\|_\infty<c<r$
	for positive solutions of \eqref{eq:4order}. Consequently, the arguments in \cite{Korman2004} retain their validity for the present problem.
An additional challenge concerns the characterization of the endpoint of the solution curve. To this end,  we adopt an idea from
P.~Lauren\c{c}ot and C.~Walker \cite[Theorem 2.20]{Laurencot2014}, where the endpoint of the solution curve was well-studied for problem \eqref{eq:biharmonic}.
Combining the idea and our a priori estimate of $C^3$-norm, we derive an explicit bounded function as the endpoint of the global solution curve.
The main contributions of this paper are the derivation of a priori estimates,
the extension to a broader class of nonlinearities, and the application of these results to new models from  \cite{Koochi2020}.

\begin{example}
Theorem \ref{th2:u bounded} applies to many doubly-supported beam-type MEMS/NEMS models arising from the recent monograph \cite[Chapter 2]{Koochi2020}
when the boundary conditions are of clamped-clamped type (cf. \cite[(2.126)]{Koochi2020}).
Some typical governing equations are present as follows.
	\begin{enumerate}
		\item Carbon-nanotube actuator in NEMS (cf. \cite[(2.147)]{Koochi2020})
	$$
	u''''=\frac{\beta_{n}}{(1-u)^{n}}.
	$$
	\item  Size-dependent double-sided nanobridge with single nanowire (cf. \cite[ (2.202)]{Koochi2020})
	$$	u''''=\frac{\beta_{vdW}}{k(1+\delta) (1-u)^{4}}+\frac{\alpha}{(1+\delta)(1-u)\ln^2[2k(1-u)]}.\label{eq:vondelwal1}$$
	\item Size-dependent double-sided nanobridge with two nanowires (cf. \cite[(2.269)]{Koochi2020})
	$$u''''=\frac{\beta_{vdW}}{ 2(1+\delta)(1-2u)^{5/2}}+\frac{\alpha}{2(1+\delta)(1-2u)\ln^2[k(1-2u)]}.$$
	\item Size-dependent nanoactuator (cf. \cite[(2.103)]{Koochi2020})
	$$ u''''=\frac{\beta_{Cas}}{(1+\delta)(1-u)^{4}}+\frac{\alpha}{(1+\delta)(1-u)^{2}}+\frac{\alpha \gamma}{(1+\delta)(1-u)}.
	$$
	\end{enumerate}
Here, $\beta_{n}/\beta_{vdW}$, $\delta$, $\alpha$, and $\beta_{Cas}$ are the dimensionless parameters of the van der Waals force,  for incorporating the size effect, associated with the external voltage, of the Casimir force, respectively; $k$ indicates the gap to nanowire radius ratio; $\gamma$ is related to the
gap-to-width ratio associated with the fringing field effect.

As the analysis on fulfillment of conditions for these examples is the same as in \cite{Liu2022} for hinged-hinged boundary conditions, we omit the details and refer the readers to \cite[Section 2]{Liu2022}.
\end{example}	
	
The paper is organized as follows. In Section 2, we recall some preliminary results and prove several key lemmas. Section 3 presents the proof of the main theorem. Finally, Section 4 provides concluding remarks and suggests some open problems for future work.

\section{Lemmas}\label{sec:lemmas}
In this section, we prove some a priori estimates, which play crucial roles in the proof of Theorem  \ref{th2:u bounded}.
	
First, we review several facts about problem \eqref{eq:4order} that have been proven in P.~Korman \cite{Korman2004} for the case $r=+\infty$, but clearly hold after modifying the range from $r=\infty$ to $r<\infty$. Specifically, assuming $f(u) \in C^1[0, r)$ satisfies \eqref{eq:f positive} and \eqref{eq:f increasing} for some $r<\infty$, we derive the following facts:
\begin{enumerate}
	\item[({A})] (Linearization)
	According to \cite[Corollary 2.2]{Korman2004}, the linear space of the non-trivial solutions of the linearized problem
	\begin{equation}\label{eq:liner2}
	\left\{\begin{array}{l}
	w^{\prime \prime\prime \prime}(x)=\lambda f^{\prime}(u(x))w(x), \quad  x \in(0, 1),  \\
	w(0)=w^{\prime}(0)=w(1)=w^{\prime}(1)=0.
	\end{array}\right.
	\end{equation}
is either empty or one-dimensional. Furthermore, according to \cite[Theorem 2.13]{Korman2004},  $w$ cannot vanish inside $(0,1)$, i.e, the sign of any non-trivial solution of \eqref{eq:liner2} does not change.

	\item[({B})] (Convexity and inflection points)
	According to \cite[Lemma 2.3]{Korman2004}, any positive solution of \eqref{eq:4order} satisfies
	\begin{equation}\label{eq:u2>0}
		u^{\prime\prime}(0)>0\quad\text{and} \quad u^{\prime\prime}(1)>0.
	\end{equation}
	Furthermore, according to \cite[Lemma 2.7]{Korman2004},  $u(x)$ has exactly one local  maximum and exactly two inflection points.
	
	\item[({C})] (Symmetry)
According to \cite[Lemma 2.9]{Korman2004},   any positive solution of \eqref{eq:4order} is symmetric with respect to $x=\frac{1}{2}$. Moreover, $u^{\prime}(x)>0$ on $(0, \frac{1}{2})$.
	
	\item[({D})] (Global parameterization)
According to \cite[Lemma 2.10]{Korman2004},  all positive solutions of \eqref{eq:4order} are globally parameterized by their maximum values $ u(\frac{1}{2})$. Precisely, for each $p> 0$, there is at most one $\lambda> 0$ and at most one solution $u(x)$  of \eqref{eq:4order} such that $u(\frac{1}{2})=p$.
\end{enumerate}

	Next, we establish a priori estimates, which play a key role in the proof of the main theorem.

\begin{lemma}\label{lem:u bounded}
	Assume that $f(u) \in C^1[0, r)$ satisfies
	\eqref{eq:f positive},\eqref{eq:f increasing} and \eqref{eq:supper2} for some $r<\infty$. If $I \subset [0,\infty)$ is a bound interval, then there exists a positive constant $C$  such that any positive solution $u_{\lambda}$ of \eqref{eq:4order} with $\lambda\in I$ satisfies
\begin{equation}\label{eq:u bound0}
			\left \|u_\lambda(x)\right \|_{C^{3}[0,1]} \le C.
	\end{equation}
If further $I \subset (0,\infty)$ is a compact interval, then there exist two positive constants $c$ and $C_1$  such that any positive solution $u_{\lambda}$ of \eqref{eq:4order} with $\lambda\in I$ satisfies
	\begin{equation}\label{eq:u bound}
	\left \|u_\lambda(x)\right \|_{C[0,1]} \leq c <r\quad\text{and}\quad
		\left \|u_\lambda(x)\right \|_{C^{4}[0,1]} \le C_1.
	\end{equation}
\end{lemma}

\begin{remark}
For problem \eqref{eq:biharmonic}, an a priori bound for $C^{\frac{3}{2}}$-norm of solutions has been established in \cite[Lemma 2.11]{Laurencot2014}. In terms of H\"{o}lder regularity, the a priori bound \eqref{eq:u bound0} is optimal in the sense that for any $\alpha\in (0,1)$, there exist a sequence of $\lambda \rightarrow 0$ and a function $w\in C^{2+\alpha}[0,1]\setminus C^{3}[0,1]$ such that $u_\lambda\rightarrow w $ in $ C^{2+\alpha}[0,1]$; see Lemma \ref{lem:endpoint} below for details.
\end{remark}

\begin{proof}[\textbf{Proof}]
	For each given $\lambda>0$, denote by $u_\lambda$ positive solutions of  \eqref{eq:4order} if exist.
We claim that
	\begin{equation}\label{eq:solutionbounded}
		u_\lambda(x)<r \quad \text{for all } x \in (0,1).
	\end{equation}
	Otherwise, there exists $x_{0} \in(0,1)$ such that $ u_{\lambda}(x_{0})=r$ (Take $x_0$ to be the smallest such number). Then, it follows from  \eqref{eq:supper2} that $ \lim_{x \to x_{0}^-} f(u_\lambda(x)) = \infty$. This contradicts the fact that $u_\lambda\in C^{4}(0,1)$ satisfies the equation in \eqref{eq:4order}.
	
	From now on, we omit the subscript of $u_\lambda$ for simplicity.
			As mentioned in ({C}) above,
since \eqref{eq:f positive} and \eqref{eq:f increasing} hold,
	any positive solution $u(x)$ of \eqref{eq:4order} is symmetric with respect to $x =\frac{1}{2}$ and $u^{\prime}(x)>0\text{ on } (0,\frac{1}{2})$. Then $u(x)$ takes its maximum at $\frac{1}{2}$ and $u^{\prime}(\frac{1}{2})= 0=u^{\prime\prime\prime}(\frac{1}{2})$. It follows from \eqref{eq:f positive} that $u^{\prime\prime\prime}$ is increasing in $[0,1]$ and $u^{\prime\prime\prime}(x)< 0$ on  $[0,\frac{1}{2})$.
	In what follows,  it suffices to consider the problem on the interval $[0,\frac{1}{2}]$.
	
	 We next divide the proof into three steps.
	
	\textbf{Step 1.}	
	We claim that for any given $x\in (0,\frac{1}{2})$,
		$u^{\prime\prime}(x) $ is bounded for all $ \lambda \in I$.
	
	 Without loss of generality,  let $x=\frac{1}{4}$. We next prove that
	$u^{\prime\prime}(\frac{1}{4}) $ is uniformly bounded for $ \lambda \in I$.
		Assume on the contrary, that there exist a sequence of unbounded numbers $u^{\prime\prime}(\frac{1}{4})$ along some sequence of $\lambda_l\in I$.

	On the one hand, if $u^{\prime\prime}(\frac{1}{4}) $ is unbounded from above, we may assume that $u^{\prime\prime}(\frac{1}{4}) $ is positive.
	Since $ u^{\prime\prime\prime}(x) < 0$  on $(0,\frac{1}{2})$,
		we have
$$ u^{\prime}(x)
=\int_{0}^{x}u^{\prime\prime}(t)\,\mathrm{d}t+u^{\prime}(0)
=\int_{0}^{x}u^{\prime\prime}(t)\,\mathrm{d}t
>u^{\prime\prime}(\frac{1}{4})x\quad \text{for }  x\in (0,\frac{1}{4}).$$
Furthermore,
since $u^{\prime}(x)> 0$ on $(0,\frac{1}{2})$, it follows that
$$ u(x)
=\int_{0}^{x}u^{\prime}(t)\,\mathrm{d}t+u(0)
=\int_{0}^{x}u^{\prime}(t)\,\mathrm{d}t
>\frac{1}{2}u^{\prime\prime}(\frac{1}{4})x^2\quad \text{for }   x \in(0,\frac{1}{4}), $$
which implies that $u(x)$ is positive and unbounded on $(\frac{1}{8},\frac{1}{4})$, contradicting the boundedness \eqref{eq:solutionbounded}.
 So $ u^{\prime\prime}(\frac{1}{4}) $ is bounded from above.

On the other hand, if $u^{\prime\prime}(\frac{1}{4}) $ is unbounded from below, we may assume that $u^{\prime\prime}(\frac{1}{4}) $ is negative.
By the same way as above, we obtain that $u(x)$ is positive and unbounded on $ (\frac{3}{8},\frac{1}{2}) $, contradicting the boundedness  \eqref{eq:solutionbounded}.
So $ u^{\prime\prime}(\frac{1}{4}) $ is also bounded from below. The claim is true.

\textbf{Step 2.} We prove the a priori bound \eqref{eq:u bound0}.

We claim that   $u^{\prime \prime \prime}(0)$  is bounded for all $\lambda\in I$.
Since $ u^{\prime \prime\prime}(0) <0$, it suffices to prove that $u^{\prime \prime\prime}(0)$ is bounded from below. Assume, on the contrary, that
there exist a sequence of unbounded numbers $u^{\prime\prime\prime}(0)$ along some sequence of $\lambda_l\in I$.
Since
$$ (u^{\prime\prime \prime})''(x)=(u'''')'(x)=\lambda f^{\prime}(u(x))u^{\prime}(x)\quad \text{ on } (0,\frac{1}{2}),$$
it follows from \eqref{eq:f increasing} and ({C}) that $u^{\prime\prime \prime}(x)$ is convex on $(0,\frac{1}{2})$
and hence
$$u^{\prime \prime \prime}\Big(0+(1-\theta)\frac{1}{2}\Big) \leq \theta u^{\prime \prime \prime}(0)+(1-\theta) u^{\prime \prime \prime}\Big(\frac{1}{2}\Big)=\theta u^{\prime \prime \prime}(0)<0 \quad \text{for any } \theta \in(0, 1). $$
That is,
for any given $ \gamma \in(0, \frac{1}{2})$,  $u^{\prime \prime\prime}(\gamma)$ is negative and unbounded from below.
Since $u^{\prime \prime\prime}(x)$ is increasing on $(0,\frac{1}{2})$, it follows that $u^{\prime\prime\prime}(x)$ is unbounded from below on any proper subinterval $[\beta,\gamma]\subset\left(0, \frac{1}{2}\right)$.
Therefore,
\begin{equation}\label{eq:Integral}
\int_{\beta}^{\gamma} u^{\prime \prime \prime}(t) \,\mathrm{d}t<(\gamma-\beta)u^{\prime \prime\prime}(\gamma)<0 \;\text{ is unbounded from below}.
\end{equation}
But, the boundedness of $ u^{\prime\prime}(\beta) $ and $u^{\prime\prime}(\gamma)$ for any $\lambda \in I$, as shown in Step 1, implies
that
\begin{equation}\label{eq:boundintegral}
 \int_{\beta}^{\gamma} u^{\prime \prime \prime}(t) \,\mathrm{d}t=u^{\prime\prime}(\gamma)-u^{\prime\prime}(\beta) \;\text{ is bounded},
\end{equation}
contradicting \eqref{eq:Integral}. So the claim is true.

Due to the symmetry of $u$, $u^{\prime \prime \prime}(1)$  is also bounded for all $\lambda\in I$. Then the monotonicity of $u^{\prime\prime\prime}$ on $[0,1]$ yields
the boundedness of $\|u^{\prime\prime\prime}\|_{C^0}$ for all $\lambda \in I$.
Together with the bound \eqref{eq:solutionbounded} and the boundary conditions, it follows that the bound \eqref{eq:u bound0} holds.
	
	\textbf{Step 3.}	We prove the a priori bounds in \eqref{eq:u bound} provided that $I \subset (0,\infty)$ is compact.
	
	Since $f$ satisfies conditions \eqref{eq:f positive} and \eqref{eq:supper2}, it follows that there exists a constant $a>0$, such that
	\begin{equation}\label{eq:f>a}
	f(u)\geq \frac{a}{r-u} \qquad \text{for all } u \in (0, r).
	\end{equation}
	Indeed, set $a_1:=\liminf_{u \rightarrow r^{-}} \,(r-u)f(u)$. Then \eqref{eq:supper2} implies that $0<a_1\leq\infty$.
	If $a_1<\infty$, then for any given $\varepsilon\in (0, a_1)$, there exists a number $\delta\in(0,r)$ such that $(r-u)f(u)>a_1-\varepsilon$ on
	$(r-\delta,r)$.
	Since $(r-u)f(u)$ is continuous in $[0,r-\delta]$, we define $a_2:=\min_{[0,r-\delta]}(r-u)f(u)$ and
	$a:=\min\{a_1-\varepsilon,a_2\}$. Then \eqref{eq:f positive} implies that $a_2>0$ and hence $a>0$.
   The case of $a_1=\infty$ is similar by replacing $a_1-\varepsilon$ with some $M>0$.
So \eqref{eq:f>a} holds.
	
	Multiplying the equation in \eqref{eq:4order} by $u^{\prime} $ and integrating over $(0,x)$ by parts,
	we derive the energy identity
	$${u}^{\prime} {u}^{\prime \prime \prime}-\frac{1}{2} {u}^{\prime \prime 2}-\lambda {F}({u})=-\frac{1}{2} {u}^{\prime \prime}(0)^2\qquad \text{ for all }x \in [0,1],$$
	where $ {F}({u})= \int_{0}^{u} f(t)  \,\mathrm{d}t. $
	
	By Step 3, we have the a priori bound  $\|u_\lambda(x) \|_{C^{3}[0,1]} \le C$.
Combining the energy identity, we get that
	$$
	\lambda F(u)=u^{\prime} u^{\prime \prime \prime}-\frac{1}{2} u^{\prime \prime 2}+\frac{1}{2} u^{\prime \prime}(0)^2 \leq M.
	$$
	where $M$ is a positive constant.
	Furthermore, it follows from \eqref{eq:f>a} that
		$$
	M \geq \lambda F(u)=\lambda \int_0^u f(t) \,\mathrm{d}t \geq \lambda \int_0^u \frac{a}{r-t} \,\mathrm{d}t=- \lambda a \ln (r-t)|_0 ^u ,
	$$
	which implies that
	$$ u \leq r(1-e^{-\frac{M}{a \lambda_*}}).$$
	Here, $\lambda_*$ is the positive lower bound of the compact interval $I$.
	Letting ${c}=r(1-e^{-\frac{M}{a \lambda_*}})$,
	we obtain that $u(x)\le c<r$ on $(0,1)$ uniformly for $\lambda\in I$.
	
	Now, since $f(u)$ is a continuous function on the interval $[0,c]$ and $u''''=\lambda f(u)$, it follows that $u''''$ is uniformly bounded on $ [0,1]$ for all $\lambda\in I$.
	Combining the boundedness of $u$ and $u''''$ with the boundary conditions, we obtain that $\left \| u \right \| _{C^{4}} $ is bounded.
\end{proof}

We next give an upper bound of $\lambda$ for the existence of positive solutions of \eqref{eq:4order} with a singular nonlinearity.

\begin{lemma}\label{lem:l2}
	Assume that  $f(u)\in C[0,r) $  satisfies conditions \eqref{eq:f positive} and \eqref{eq:supper2}  for some $r<\infty$. Then there exists $\lambda_0 >0$ such that  problem \eqref{eq:4order} has no positive solution for $\lambda>\lambda_0$.
\end{lemma}

\begin{proof}
	Since the line $ y=\frac{4}{r^{2}}x $ passing through the origin is a tangent below the curve $ y=\frac{1}{r-x} $, it follows from \eqref{eq:f>a} that
	\begin{equation}\label{eq:f>2}
	f(u)\geq \frac{a}{r-u} \geq \frac{4a}{r^{2}}u \qquad \text{for all } u \in (0, r).
	\end{equation}
	
	With condition \eqref{eq:f>2} in place, the process that follows is routine. Let  $\mu_1>0$  and $\varphi_1(x)>0$ on $(0,1)$ be the principal eigenpair of the problem
	$$\left\{\begin{array}{l}
	\varphi^{\prime \prime \prime \prime}=\mu \varphi \quad \text { in }(0,1); \\
	\varphi(0)=\varphi^{\prime}(0)=0=\varphi(1)=\varphi^{\prime}(1).
	\end{array}\right.
	$$
	Let $u$ be a positive solution of \eqref{eq:4order} with some $\lambda$.
	Multiplying the equation in \eqref{eq:4order} by $\varphi_1$ and integrating over $(0,1)$,  we obtain from
	\eqref{eq:f>2} that
	$$
	\mu_1 \int_0^1 u \varphi_1 \mathrm{d} x= \int_0^1 u \varphi''''_1 \mathrm{d} x = \int_0^1 u'''' \varphi_1 \mathrm{d} x =\lambda \int_0^1 f(u)\varphi_1 \mathrm{d} x \geq \lambda \frac{4a}{r^{2}} \int_0^1 u \varphi_1 \mathrm{d} x.
	$$
It follows that
	$$
	\lambda \leq \frac{r^{2}}{4a}\mu_1.$$
	Therefore, the set of $\lambda$ for which problem \eqref{eq:4order} admits a positive solution is bounded from above.
 Let $\lambda_0$ denote its supremum, and the conclusion follows.
\end{proof}

The following result plays a crucial role in proving the endpoints of the solution curve.
\begin{lemma}\label{lem:endpoint}
Assume that  $f(u)\in C^1[0,r) $  satisfies conditions \eqref{eq:f positive}, \eqref{eq:f increasing} and \eqref{eq:supper2} for some $r<\infty$.
Let $\alpha \in (0, 1)$ and let $\{(\lambda_n,u_n)\}$ be a sequence of solutions of problem \eqref{eq:4order} with $\lambda_n\rightarrow 0$. Then there exist
a subsequence of $\{u_n\}$ (still denoted by $u_n$) and a function $w(x)$
such that
		\begin{equation}\label{eq:endpoints}
	\lim _{n \rightarrow \infty}\|u_n-w\|_{C^{2+\alpha}[0,1]}=0.
	\end{equation}
	Moreover, either $w \equiv 0$ or $\max _{x \in[0,1]} w(x)=r$.
	In the latter case,  $w\in  ( {C}^{2+\alpha}[0,1]\setminus{C}^{3}[0,1])\cap C^{4}( [0,1]\setminus \{ \frac{1}{2}\}  )$ for any $\alpha\in(0,1)$, and $w$ satisfies
\begin{equation}\label{eq:singulareq}
	\left\{\begin{array}{l}
		w^{\prime \prime \prime \prime}(x)=0, \quad x \in[0, \frac{1}{2}) \cup(\frac{1}{2}, 1]; \\
		w(0)=w^{\prime}(0)=0=w(1)=w^{\prime}(1); \\
		w(\frac{1}{2})-r=0=w^{\prime}(\frac{1}{2}),
	\end{array}\right.
	\end{equation}
which is solved uniquely by
	\begin{equation}\label{eq:explicit}
	w(x)= \begin{cases}-16 r x^3+12 r x^2, & x \in[0,\frac{1}{2}]; \\
-16 r (1-x)^3+12 r (1-x)^2, & x \in[\frac{1}{2}, 1].
\end{cases}
	\end{equation}
\end{lemma}

\begin{proof}
The existence of $w \in C^{2+\alpha}[0,1]$ and the convergence relation \eqref{eq:endpoints} follow directly from the a priori bound \eqref{eq:u bound0},
due to the compact embedding $C^{3}[0,1]\hookrightarrow C^{2+\alpha}[0,1]$ for any $\alpha \in [0, 1)$.
As a limit of $C^{2+\alpha}$-convergence, $w$ satisfies
\begin{equation}\label{eq:wbvc}
 w(0)=w^{\prime}(0)=w(1)=w^{\prime}(1)=0=w^{\prime}(\tfrac{1}{2}),
\end{equation}
since every solution satisfies the boundary conditions and the symmetric property.

In view of \eqref{eq:solutionbounded}, we have $ w(x)\leq r$. We divide the discussion into two cases.

\textbf{Case 1: $\max _{x \in[0,1]} w(x)<r$.} We claim that $w \equiv 0$ for all $x \in[0,1]$.

Since in this case $w(x)<r$ for all $x \in[0,1]$, it follows from \eqref{eq:endpoints} that $f(u_n)$ is bounded on $[0,1]$
and $\lambda_n f(u_n) \rightarrow 0$ as $\lambda_n \rightarrow 0$.
Since every solution $(\lambda_n,u_n)$ of \eqref{eq:4order} satisfies the integral form
		\begin{equation}\label{eq:integral}
		{u}_{n}^{\prime \prime}(x)  = {u}_{n}^{\prime \prime }(0) + {u}_{n}^{\prime \prime \prime }(0)x+ {\int}_{0}^{x}{\lambda }_{n}f( {{u}_{n}( \xi ) })(x-\xi) \,\mathrm{d}\xi, \quad
 x \in[0,1].
	\end{equation}
Passing to the limit as $\lambda_n\rightarrow 0$, we conclude from \eqref{eq:u bound0} and \eqref{eq:endpoints} that
	$$
	w^{\prime\prime}(x)=w^{\prime \prime}(0) +\vartheta x,\quad  x \in[0,1],
	$$
where $\vartheta$ is a constant.
Clearly, $w(x) \in C^4[0,1]$ and $w^{\prime \prime \prime \prime}(x)\equiv 0$ on $[0,1]$. It follows from \eqref{eq:wbvc} that the claim is true.

\textbf{Case 2: $\max _{x \in[0,1]}  w(x) = r$.} We prove that $x=\frac{1}{2}$ is the unique maximum point of $w$.

	By the symmetry and the monotonicity of $u_n$ as given in ({C}), it is clear that $w(x) $ is symmetric on $[0,1]$, $ w(\tfrac{1}{2}) = r $, $w(x) $ is non-decreasing on the interval $ (0,\tfrac{1}{2})$ and non-increasing on $( \tfrac{1}{2},1) $.
	So there exists a number $ a \in (0,\tfrac{1}{2}]$ such that $w(x) = r $ for all $x\in [a, 1-a]$
	and $w(x) < r $ for all $x\in [0,a)\cup(1-a,1]$. Moreover, $ w(a) = r= w(1-a)$ and $ w'(a)= 0= w'(1-a)$.
	
For any $ \rho \in (0,a) $,
since  $w(x)<r$ for all $x \in[0,\rho]$, it follows from \eqref{eq:endpoints} that $f(u_n)$ is bounded on $[0,\rho]$.
Similar to Case 1, we have 
\begin{equation}\label{eq:twoseg}
w(x) \in C^4 ([0,a)\cup(1-a,1]) \;\text{ and }\;  w^{\prime\prime\prime\prime}(x) = 0 \; \text{ for all }  x \in [0,a)\cup(1-a,1].
\end{equation}

	We claim that $ w(x) \in C^3([0,\frac{1}{2})\cup (\frac{1}{2},1]) $. In fact, since $ \|u_{n}\|_{C^3[0,1]} $ is bounded and $u^{\prime\prime\prime\prime}_{n}(x)$ is positive and increasing on $(0,1)$, using the arguments as in \eqref{eq:Integral} and \eqref{eq:boundintegral} (Replace $u^{\prime\prime\prime}$ with $u^{\prime\prime\prime\prime}$),
we deduce that for any given  $x\in[0,\frac{1}{2})$, the sequence $u_n^{\prime\prime\prime\prime}(x)$ is bounded and further for any closed subinterval $[\beta,\gamma] \subset [0,\frac{1}{2})$,
$\|u^{\prime\prime\prime\prime}_n\|_{C^0[\beta,\gamma]}$ is bounded. Together with \eqref{eq:solutionbounded}, it follows that $\|u_{n}\|_{C^4[\beta,\gamma]}$ is bounded.
This implies that  $w(x) \in C^3[\beta,\gamma]$ by the compact embedding $C^4[\beta,\gamma]\hookrightarrow C^3[\beta,\gamma]$.
Due to the arbitrariness of $[\beta,\gamma] \subset [0,\frac{1}{2})$, we conclude that $w(x) \in C^3[0, \frac{1}{2})$. Similarly, we have that $w(x) \in C^3(\frac{1}{2},1]$ and the claim is true.

Using an idea from \cite[(2.46)]{Laurencot2014},
we next show that $a=\frac{1}{2}$. Suppose on the contrary that $ a<\frac{1}{2}$. By the claim above, we have
\begin{equation}\label{eq:derivatives}
  0=w(a)-r=w^{\prime}(a)=w^{\prime\prime}(a)=w^{\prime\prime\prime}(a).
\end{equation}
Multiplying \eqref{eq:twoseg} by $ w $, integrating over $ (0,a)$ and using \eqref{eq:derivatives},
	we obtain that
	$$0 = \int_0^{a} w \cdot w'''' \mathrm{d}x  = \int_0^{a} (w'')^2 \mathrm{d}x,$$
	which implies that $ w''(x) \equiv 0 $ on $[0,a] $ and hence $ w(x) = k_1 x + k_2 $. Here, $k_1, k_2$ are constants.
	It follows from \eqref{eq:wbvc} that $ w(x) \equiv 0 $ for $ x \in [0,a] $,
	contradicting \eqref{eq:derivatives}.  
	So $ a = \frac{1}{2} $.

   Consequently, from \eqref{eq:twoseg} we obtain that $w\in C^{4}( [0,1]\setminus \{ \frac{1}{2}\} )$.
Direct integration shows that the explicit function $w(x)$ in \eqref{eq:explicit} uniquely satisfies \eqref{eq:singulareq}. Notably, $w'''(\frac{1}{2})$ is undefined because the third derivatives from the left and right sides are different. Thus, $w\notin C^3[0,1]$.
	\end{proof}

\section{Proof of Theorem \ref{th2:u bounded}}

In this section, we prove Theorem \ref{th2:u bounded}.
  Let us recall a well-known local bifurcation theorem due to Crandall and Rabinowitz \cite[Theorem 3.2]{Crandall1973}.
 \begin{theorem}[\cite{Crandall1973}] \label{th:C-R}
Let $X$ and $Y$ be Banach spaces. Let $ (\bar{ \lambda}  ,\bar{x}   )\in  \mathbb{R} \times X $ and $F$ be a continuously differentiable mapping of an open neighborhood of $ (\bar{ \lambda}, \bar{x})$ into  $Y$.
	Let the null-space  $N(F_{x} (\bar{ \lambda},\bar{x})) = \mathrm{span} \{ x_{0}   \}  $ be a one-dimensional and $\mathrm{codim}R(F_{x}  (\bar{ \lambda}  ,\bar{x}   ))$ $=1$.
	Let $F_{\lambda} ( \bar{\lambda },\bar{x}  ) \notin R( F_{x} ( \bar{\lambda },\bar{x} )  )$. If $Z$ is the complement of $\mathrm{span} \{ x_{0}   \} $ in $X$,
then the solution of $F ( \lambda,x  )=F ( \bar{\lambda },\bar{x}  )$ near $ ( \bar{\lambda },\bar{x}  )$ forms a curve
	$ ( \lambda(s),x(s)  )= ( \lambda+\tau (s),\bar{x}+sx_{0}+z(s) )$,
	where $s\to  ( \tau (s),z(s)  ) \in \mathbb{R}\times X $ is a function that is continuously differentiable near $s=0$ and $ \tau (0) = \tau ^{\prime}(0) = 0, z(0)=z^{\prime}(0)=0.$ Moreover,
	if $F$ is $k$-times continuously differentiable, so are $\tau(s)$ and $z(s)$.
\end{theorem}

\begin{proof}[\textbf{Proof of Theorem \ref{th2:u bounded}}]
Based on the facts ({A})--({D}) and the lemmas in Section \ref{sec:lemmas}, the proof closely follows the original proof of Theorem \ref{thm:korman},
employing the Implicit Function Theorem and Theorem \ref{th:C-R}.
To avoid repetition, we omit the details and refer readers to \cite{Korman2004}.
The new part of our result, concerning the endpoint $w$, follows from Lemma \ref{lem:endpoint}.

For the convenience of the reader, we here briefly outline the main ideas and key points of the proof.
Consider the Banach spaces $X=\{u\in C^4[0,1]\mid u(0)=u(1)=u^{\prime}(0)=u^{\prime}(1)=0\}$ and $Y=C[0,1]$.
Similar to \cite[(2.4)]{Laurencot2014}, we define $F: \mathbb{R}\times \mathcal{O} \rightarrow Y$ by $F(\lambda, u)=u''''-\lambda f(u)$,
where $\mathcal{O}=\{u\in X\mid u<r \text{ in } [0,1]\}$.
Given $f(0)>0$, we may smoothly extend $f$ to a non-negative $C^2$ function on $(-\infty,r)$ and
 the extension does not affect the solution set of positive solutions.
Note that for non-negative $f$, by \cite[Theorem 1]{chow1973maximum}, any non-trivial solution of \eqref{eq:4order} is positive on $(0,1)$.
Starting at the point of the trivial solution $(\lambda = 0,u = 0)$, we derive the desired solution curve $\mathcal{S}$ by the continuation approach, relying on the Implicit Function Theorem (at regular points) and the Crandall-Rabinowitz theorem (at possible singular or turning points) to smoothly continue the curve. While Lemma \ref{lem:l2} indicates that
the solution curve cannot continue infinitely in the direction of increasing $\lambda$, a priori bounds \eqref{eq:u bound} in Lemma \ref{lem:u bounded}
implies that this curve cannot stop at any $\lambda>0$, nor can it have a vertical asymptote at any $\lambda>0$.
Furthermore, by the formula of the bifurcation direction at singular points, this curve must turn left at each possible singular point because of the convexity of $f$.
Therefore, the solution curve continues globally and admits exactly a turn at some
critical point $(\lambda_0,u_0)$. After bending back at $(\lambda_0,u_0)$,
Lemma \ref{lem:endpoint} states that when $\lambda \downarrow 0$,
there are two possible behaviors for the solution curve: it converges to either $(0,0)$ or $(0,w)$.
However, the uniqueness of solutions near the origin $(0,0)$ excludes the convergence to $(0,0)$, according to the Implicit Function Theorem.
Here, $w(x)$ is explicitly given by \eqref{eq:explicit} and its maximum value is $r$.
According to (D), all positive solutions are globally parameterized by $p:=u(\frac{1}{2})=\left \| u \right \|_\infty$.
From the solution curve $\mathcal{S}$, we immediately obtain a smooth global bifurcation curve
$$\mathcal{C}=\{(\lambda,\left \| u \right \|_\infty) \mid \lambda \text{ and } u \text{ satisfy } \eqref{eq:4order} \}.$$
Along this curve, the parameter $p$ monotonically increases; see Figure \ref{fig:fig1}(ii).
Moreover, since the curve parameterizes all solutions as $p$ ranges from $0$ to $r$,  it follows that there are no additional solution branches.
\end{proof}

\section{Concluding Remarks}
In this paper, we have established a global bifurcation result for the fourth-order equation with doubly-clamped boundary conditions, assuming the nonlinearity $f$
is increasing and convex. We have derived the complete structure of the solution set, revealing the exact multiplicity of positive solutions. The corresponding bifurcation diagram is depicted in Figure \ref{fig:fig1}(ii). Examples of fourth-order MEMS models arising from the recent monograph \cite{Koochi2020} have been presented to illustrate applications of the main theorem. Additionally, we have built the a priori estimate $\|u\|_{C^3}<C$, which is optimal in terms of H\"{o}lder regularity. Based on the crucial estimate, we have demonstrated that the regular solutions converges to an explicit bounded function in $ {C}^{2+\alpha}\setminus C^3[0,1]$ along the upper branch of the solution curve as $\lambda\rightarrow 0$.

We suggest two promising topics for further research.
\begin{enumerate}
	\item Consider a more general MEMS model than \eqref{eq:4order}:
$$
\left\{\begin{array}{l}
u^{\prime\prime\prime\prime}(x)-Tu^{\prime\prime}(x)=\lambda f(u(x)), \quad x\in (0,1),T\geq 0;
\\
u(0)=u(1)=0=u^{\prime}(0)=u^{\prime}(1).
\end{array}\right.
$$
Some interesting problems remain to be addressed in
establishing analogous results to Theorem \ref{th2:u bounded} for the cases when $T> 0$ and without
the assumption $f'>0$.
Major difficulties include establishing suitable a priori bounds and achieving global parametrization of the solution curve via its maximum.	
	\item Consider the fourth-order regularized MEMS model
	arising in \cite[(3.13b)]{Lindsay2014}:
	$$
	\left\{\begin{array}{l}
	u^{\prime \prime \prime \prime}(x)=
	\frac{\lambda}{(1-u)^{2}}-\frac{\lambda\varepsilon^{m-2} }{(1-u)^{m}}, \quad x \in(0,1); \\
	u(0)=u(1)=0=u^{\prime}(0)=u^{\prime}(1),
	\end{array}\right.
	$$
	where $\varepsilon>0$ and $m>2$. Compared to the known increasing and convex nonlinearity,  $f(u)=\frac{1}{(1-u)^{2}}-\frac{\varepsilon^{m-2} }{(1-u)^{m}}$ here is a non-monotonic and convex-concave function. The study of global bifurcation curves becomes more challenging. In contrast to the $\supset$-shaped curve when $\varepsilon=0$, numerical bifurcation diagrams in \cite[Figure 4]{Lindsay2014} for $m=4$ exhibit $S$-shaped curves appearing for small positive values of $\varepsilon$, but the strict proof remains unknown.
The present Theorem \ref{thm:korman} might be helpful in solving this problem.
\end{enumerate}

\section*{Acknowledgment}
The second author is partially supported by Guangdong Basic and Applied Basic Research Foundation (Grant No.2022A1515011867), which is gratefully acknowledged.

\end{document}